\documentclass[11pt]{article}
\usepackage[letterpaper,margin=0.9in]{geometry}
\usepackage[T1]{fontenc}
\usepackage{lmodern,amsmath,amssymb,amsthm,microtype}
\usepackage[hidelinks]{hyperref}
\newcommand{\C}{\mathbb C}
\newcommand{\Z}{\mathbb Z}
\newcommand{\D}{\Delta}
\newcommand{\Dclosed}{\overline D}
\newcommand{\B}{\overline B}
\newcommand{\OO}{\mathcal O}
\newcommand{\im}{\operatorname{Im}}
\newcommand{\diag}{\operatorname{diag}}
\newcommand{\PP}{\mathbb P}
\newcommand{\pr}{\operatorname{pr}}
\newcommand{\dist}{\operatorname{dist}}
\newtheorem{theorem}{Theorem}[section]

\newtheorem{lemma}[theorem]{Lemma}
\newtheorem{corollary}[theorem]{Corollary}
\theoremstyle{remark}
\hypersetup{pdftitle={Oka complements in algebraic tori and their logarithmic covers},pdfauthor={Yun-Heng Du, Bin Guo, Peng-Chao Wang, Song-Yan Xie}}
\title{Oka complements in algebraic tori and their logarithmic covers}
\author{Yun-Heng Du \and Bin Guo \and Peng-Chao Wang \and Song-Yan Xie}
\date{October 11, 2026}
\begin{document}
\maketitle

\begin{abstract}
Suppose $n\ge2$ and $K\subset\C^n$ is compact and polynomially convex. Write
$q(z)=(e^{-iz_1},\ldots,e^{-iz_n})$,
$X_K=(\C^*)^n\setminus(K\cap(\C^*)^n)$, and
$Y_K=\C^n\setminus q^{-1}(K)$. We prove that $X_K$ and its logarithmic covering
$Y_K$ are Oka. More strongly, $Y_K$ admits normalized holomorphic families of
entire embeddings through prescribed holomorphic sections. The proof combines
polynomial pair Hamiltonian transport in multiplicative coordinates, compatible
with the coordinate hyperplanes, with a fixed Fatou--Bieberbach model.
Applications include Oka slab crosses over arbitrary nonempty connected planar
domains, complements of periodic equal-radius Euclidean balls centered at
products of planar lattices, and complements of sufficiently small closed
metric balls in products of elliptic curves.
\end{abstract}


\section{Introduction}\label{sec:introduction}

Kusakabe~\cite[Question 5.6]{Kcomp} asks whether equal-radius balls centered at
an arbitrary discrete set $D\subset\C^n$, $n>1$, have an Oka complement when
the infimum of the distances between distinct centers is greater than twice the
radius. His Corollary~5.7 covers the case
$D=\mathbb Z\times\{0\}^{n-1}$, and Theorem~\ref{thm:balls} below treats the
case in which $D$ is a full product of planar lattices.
Forstneri\v c~\cite[p.~23, paragraph after Proposition 6.1.4]{F04} asked
whether removing sufficiently small balls from a complex torus preserves the
Oka property, following earlier results on dominability; for products of
elliptic curves with their flat product metrics, Corollary~\ref{cor:torus}
answers this for one ball.

The two questions are addressed by one construction. For broad introductions to
modern Oka theory, including its geometric and homotopy-theoretic viewpoints,
see Forstneri\v c~\cite{F17}, Forstneri\v c--L\'arusson~\cite{FL11}, and
L\'arusson~\cite{L14}. Recall that a complex manifold is Oka if it has the
convex approximation property: for every $m\ge1$, every holomorphic map from a
neighborhood of a compact convex set in $\C^m$ can be approximated uniformly on
that set by holomorphic maps from $\C^m$. We use Kusakabe's equivalent convex
spray criterion~\cite[Theorem 2.2]{Kus} and his localization
theorem~\cite[Theorem 1.4]{Kus}; the latter applies to analytic-Zariski open
neighborhoods, that is, to complements of closed complex subvarieties.

The closest published antecedent for the parameter-dependent mechanism used
here is the work of Forstneri\v c and Wold. They constructed
holomorphic families of Fatou--Bieberbach domains avoiding compact
holomorphically convex obstacles, with a section-transport step by fiber
automorphisms that remain close to the identity near the obstacle
\cite[Theorems~1.1 and~3.1; Lemma~3.2]{FW20}. Their subsequent work treats a
compact polynomially convex obstacle together with one coordinate hyperplane
\cite[Theorem~2.3 and Corollary~3.2]{FW24}. These constructions provide the
transport-and-fixed-model pattern that we adapt below.

In our earlier work~\cite{DGWX}, complements of tubes around the real plane and
products of closed annuli in $\C^2$ were shown to be Oka by constructing
holomorphic families of embeddings through prescribed holomorphic sections.
In particular, \cite[Proposition~1.2, version~1]{DGWX} gives a normalized
holomorphic family of Fatou--Bieberbach maps outside a two-dimensional tube,
together with a holomorphic family of attracting automorphisms. The
logarithmic obstacles below are unbounded and may interact with several
coordinate hyperplanes; controlling their full logarithmic preimages requires
the new displacement argument developed here.

We regard $(\C^*)^n$ as the $n$-dimensional complex algebraic torus. For
$K\subset\C^n$, write
\[
 \begin{gathered}
 q(z)=(e^{-iz_1},\ldots,e^{-iz_n}),\qquad
 E_K=q^{-1}(K),\\
 Y_K=\C^n\setminus E_K,\qquad
 X_K=(\C^*)^n\setminus\bigl(K\cap(\C^*)^n\bigr).
 \end{gathered}
\]
The restriction $q:Y_K\to X_K$ is the logarithmic covering associated with
$K$. The following theorem is the technical core of the paper.

\begin{theorem}\label{thm:general}
For $n\ge2$ and compact polynomially convex $K\subset\C^n$, the following
statements hold.
\begin{enumerate}
\item[(i)] The manifolds $Y_K$ and $X_K$ are nonempty, path connected, and Oka.
\item[(ii)] For every compact convex $L_0\subset\C^m$ and holomorphic map
$f:U_0\to Y_K$ from a neighborhood of $L_0$, there are a neighborhood
$V\subset U_0$ of $L_0$ and a holomorphic map
\[
 F:V\times\C^n\longrightarrow Y_K,\qquad
 F(\zeta,0)=f(\zeta),\qquad D_\xi F(\zeta,0)=I_n,
\]
such that every fiber map $F(\zeta,\cdot)$ is biholomorphic onto its image.
\end{enumerate}
\end{theorem}

Part~(ii) is stronger than the Oka conclusion: locally in the parameter, it
gives normalized entire embeddings through a prescribed section with
prescribed first jet. If $E_K\ne\varnothing$, their images are
Fatou--Bieberbach domains in $\C^n$ that avoid $E_K$; if $E_K=\varnothing$, the
construction is a translation with image $\C^n$. This is the form used with
Kusakabe's convex spray criterion, and it is the reusable engine behind all the
applications below.

Theorem~\ref{thm:general} extends the normalized family construction
of~\cite[Proposition~1.2, version~1]{DGWX} from two-dimensional tubes to the
logarithmic obstacles $E_K\subset\C^n$. The compact-quotient construction
of~\cite[Theorem~1.1, version~2]{DGWX} assumes a compact holomorphically convex
obstacle in $(\C^*)^n$. Relative to the compact-obstacle families
in~\cite{FW20} and the one-coordinate-hyperplane setting in~\cite{FW24}, the
additional issue here is that $K$ may meet several coordinate hyperplanes.
Thus $K\cap(\C^*)^n$ can be noncompact and the imaginary parts of $E_K$ need
not be bounded in both directions. This is the situation for periodic balls:
the closure of their coordinatewise inverted obstacle has boundary strata on
the coordinate hyperplanes, and Lemma~\ref{lem:compactification} shows that
this closure is polynomially convex.

The proof has three layers. First, polynomial pair Hamiltonian potentials in
multiplicative coordinates produce complete flows that extend across the
coordinate hyperplanes. Second, an accumulated displacement variable converts
uniform control on the unbounded logarithmic obstacle $E_K$ into approximation
on a compact neighborhood of $K$. Third, the resulting automorphisms transport
a prescribed section to a fixed Fatou--Bieberbach model; pulling the model back
and normalizing the derivative gives the family in
Theorem~\ref{thm:general}. Kusakabe's convex spray criterion then yields the Oka
property, while the logarithmic covering transfers it to $X_K$.

For nonempty connected planar domains $U_1,\ldots,U_n$, the slab cross is
\[
 \mathcal S(U_1,\ldots,U_n)
 =\bigcup_{j=1}^n\C^{j-1}\times U_j\times\C^{n-j}.
\]

Every point $z\in\mathcal S(U_1,\ldots,U_n)$ lies in a spanning tube contained
in the cross. Indeed, choose $i$ with $z_i\in U_i$, choose $k\ne i$ and
$b\in z+\C e_k$ with $b_k\in U_k$, and use the coordinate lines
$z+\C e_j$ for $j\ne i$ together with $b+\C e_i$. They form a connected
spanning tree whose sufficiently small tubular neighborhood remains in the
cross. Hence Alarc\'on and Forstneri\v c~\cite[Corollary~2.4]{AF} imply that
every slab cross is Oka-1. The following theorem strengthens this conclusion to
the full Oka property.

\begin{theorem}\label{thm:slabs}
For every $n\ge2$ and nonempty connected open sets $U_1,\ldots,U_n\subset\C$, the slab cross $\mathcal S(U_1,\ldots,U_n)$ is Oka.
\end{theorem}

The strength of the statement is that no boundedness or boundary regularity of
the domains is assumed. When
$0\in U_j$ for every $j$, the cross contains a neighborhood of each coordinate
hyperplane; in particular,
\[
 \Omega_n=\{z\in\C^n:\min_{1\le j\le n}|z_j|<1\}
\]
is Oka. For $n=2$, it is $(\D\times\C)\cup(\C\times\D)$, where $\D$
denotes the open unit disc.

The second application is to periodic Euclidean balls: for $c\in\C^d$ and
$r>0$, the closed Euclidean ball of radius $r$ about $c$ is
\[
 \B(c,r)=\{z\in\C^d:|z-c|\le r\}.
\]
For rank-two planar lattices $\Lambda_j\subset\C$, $1\le j\le n$, write
\[
 \Lambda=\Lambda_1\times\cdots\times\Lambda_n,\qquad
 \ell_j=\min_{0\ne\lambda\in\Lambda_j}|\lambda|,\qquad
 \ell=\min_j\ell_j.
\]
With respect to the product Euclidean norm, $\ell$ is the length of a shortest
nonzero vector of $\Lambda$.

\begin{theorem}\label{thm:balls}
If $n\ge2$ and $0<2r<\ell$, then
\[
 \Omega=\C^n\setminus\bigcup_{\lambda\in\Lambda}\B(\lambda,r)
\]
is Oka.
\end{theorem}

The substantive new step is the compactification of the infinite periodic
obstacle: after coordinatewise inversion, Lemma~\ref{lem:compactification}
identifies a compact polynomially convex closure with explicit boundary strata.

Equip
\[
 T=(\C/\Lambda_1)\times\cdots\times(\C/\Lambda_n)
\]
with the flat product metric induced by the Euclidean norm on $\C^n$. Its
injectivity radius is $\ell/2$, and we write $\B_T(c,r)$ for its closed metric
balls.

\begin{corollary}\label{cor:torus}
For $n\ge2$, every $c\in T$, and $0<r<\ell/2$, the complement of the closed metric ball $\B_T(c,r)$ in $T$ is Oka.
\end{corollary}

This is the quotient consequence of Theorem~\ref{thm:balls}: the injectivity
radius condition makes the full inverse image of the ball exactly the periodic
ball obstacle.

Translated inversion charts reduce both applications to
Theorem~\ref{thm:general}, with Kusakabe's localization theorem handling the
change of chart. For slab crosses the inverted obstacle is a product of compact
planar sets. For periodic balls the coordinate boundary strata are
coordinatewise inversions of balls of lower dimension, and it remains to show
that the whole compactified obstacle is polynomially convex
(Lemma~\ref{lem:compactification}).

Section~\ref{sec:approximation} contains the connectedness statement and the
approximation lemmas, Sections~\ref{sec:transport} and~\ref{sec:sprays} the
transport argument and the proof of Theorem~\ref{thm:general}, and
Sections~\ref{sec:slabs}--\ref{sec:elliptic} the applications. The fixed
two-dimensional Fatou--Bieberbach model is constructed in
Appendix~\ref{app:model}.

\section{Connectedness and polynomial approximation}\label{sec:approximation}
This section separates the preliminary tasks used in
Theorem~\ref{thm:general}. Lemma~\ref{lem:connected} supplies the topology of
the base and its logarithmic covering, while the remaining lemmas provide the
convexity and approximation tools used by the transport construction. For background on polynomial convexity and approximation,
see~\cite{Stout}.

\begin{lemma}\label{lem:connected}
For compact polynomially convex $K\subset\C^n$, both $X_K$ and $Y_K$ are nonempty and path connected.
\end{lemma}
\begin{proof}
Choose $R>0$ with $|w_j|<R$ for every $w\in K$ and every $j$, and set
$b=(R,\ldots,R)\in X_K$. Since $q$ is surjective, $Y_K\ne\varnothing$ as well.
The complement $M=\C^n\setminus K$ is connected: any component other than the
one containing the exterior of a large ball would be bounded, with boundary in
$K$, and the maximum-modulus principle would force it into $\widehat K=K$. Thus $M$
is a connected open set, hence path connected.

Inside a ball in $M$, two points outside the hyperplanes can be joined in their
complex affine line, whose intersection with the hyperplanes consists of at most
$n$ points. A finite chain of such balls along a path in $M$, with joining
points chosen outside the hyperplanes, gives a path in $X_K$, so $X_K$ is path
connected as well.

The loops
\[
 \gamma_j(t)=(R,\ldots,R e^{2\pi it},\ldots,R),\qquad
 0\le t\le1,\quad 1\le j\le n,
\]
lie in $X_K$. Concatenations of their lifts and reversed lifts join every point of
$q^{-1}(b)=b'+2\pi\Z^n$ to a fixed $b'\in q^{-1}(b)$. Any $z\in Y_K$ can first
be joined to this fiber by lifting a path in $X_K$ from $q(z)$ to $b$, and then
to $b'$. Thus $Y_K$ is path connected.
\end{proof}

The next lemma is the gluing device for the two local prescriptions in
Lemma~\ref{lem:transport}: one is imposed near an ambient compact set and the
other along a disjoint closed submanifold. It is the disjoint-submanifold case
of Forstneri\v c's convexity lemma~\cite[Lemma 6.5]{F99}; the Stein form needed
here is proved directly.
\begin{lemma}\label{lem:union}
Let $Z$ be Stein, $G\subset Z$ a closed complex submanifold, $P\subset Z\setminus G$ compact and $\OO(Z)$-convex, and $Q\subset G$ compact and $\OO(G)$-convex. Then $P\cup Q$ is $\OO(Z)$-convex.
\end{lemma}
\begin{proof}
Suppose first that $x\notin P\cup Q$ and $x\notin G$. By Cartan's theorem A
there is $h_0\in\OO(Z)$ that vanishes on $G$ but not at $x$, and by the
$\OO(Z)$-convexity of $P$ there is $a\in\OO(Z)$ with $|a(x)|>\sup_P|a|$. For
large $N$ the function $h_0a^N$ then separates $x$ from $P\cup Q$: it vanishes
on $Q$, while its modulus at $x$ exceeds its supremum on $P$.

Suppose next that $x\in G\setminus Q$. Choose $g\in\OO(G)$ with
$|g(x)|>\sup_Q|g|$, so that in particular $g(x)\ne0$, and extend $g$ to
$\widetilde g\in\OO(Z)$ with $\widetilde g(x)=g(x)$ by Cartan's theorem B. We
claim that there is $h\in\OO(Z)$ with $h|_G=1$ and $\sup_P|h|<1$; to construct
it, we use Cartan's theorem A together with the compactness of $P$ to find
finitely many global functions $h_j$ vanishing on $G$ with no common zero near
$P$. A Stein neighborhood $W$ of $P$ on which they have no common zero admits a
solution $\sum h_jb_j=1$ by Cartan's theorem B. Approximating the $b_j$ on $P$
by global functions $B_j$, which is possible by Oka--Weil, we obtain
$h=1-\sum h_jB_j$ with the required properties. Now
$\widetilde g h^N$ agrees with $g$ on $Q$ and takes the value $g(x)$ at $x$,
whereas
\[
 \sup_P|\widetilde g h^N|
 \le\bigl(\sup_P|\widetilde g|\bigr)\bigl(\sup_P|h|\bigr)^N
 \longrightarrow0.
\]
For $N$ large enough, its modulus at $x$ therefore exceeds both its supremum
on $P$ and its supremum on $Q$.
\end{proof}

The transport argument requires approximation not only of the local potentials
but also of the fiber derivatives that determine their Hamiltonian fields. The
following fixed-neighborhood form makes that simultaneous control explicit.

\begin{lemma}\label{lem:polyapprox}
Suppose $B$ is Stein and $S\subset B\times\C^n$ is compact and $\OO(B\times\C^n)$-convex. A function holomorphic near $S$ can be approximated uniformly, together with any fixed finite number of fiber derivatives, on a compact neighborhood of $S$ chosen inside its domain and fixed independently of the approximation error, by polynomials
\[
 P(b,w)=\sum_{k\in A}c_k(b)w^k,\qquad
 A\subset\Z_{\ge0}^n\text{ finite},\quad c_k\in\OO(B).
\]
\end{lemma}
\begin{proof}
Choose, and then fix, nested compact $\OO(B\times\C^n)$-convex neighborhoods of
$S$ inside the domain of the given function. Oka--Weil approximation on the
larger one produces a function in $\OO(B\times\C^n)$, and Cauchy's estimates
then control its derivatives on the smaller one. Expanding this approximant in a
Taylor series in $w$ about zero, we obtain coefficients holomorphic on $B$, and
the finite truncations converge together with all fiber derivatives uniformly on
compact subsets.
\end{proof}

\section{Polynomial transport}\label{sec:transport}
We approximate transport by compositions of complete holomorphic flows, as in
Anders\'en--Lempert theory~\cite{AL}; here the polynomial exponential modes
must also extend across the coordinate hyperplanes. The first lemma preserves
global automorphisms while discretizing a nonautonomous flow; the second uses
that mechanism to move a section while remaining uniformly close to the
identity on the possibly unbounded logarithmic obstacle.

For $1\le r<s\le n$ and a scalar potential $H$ in additive coordinates, the pair Hamiltonian field is
\[
 V_{rs,H}=(\partial_{z_s}H)e_r-(\partial_{z_r}H)e_s.
\]
If $H(z)=c\,q(z)^k$, where $k\in\Z_{\ge0}^n$, its velocity is
\begin{equation}\label{eq:mode}
 v_{rs,k}(w)=icw^k(-k_s e_r+k_r e_s),\qquad w=q(z).
\end{equation}
Since $k\cdot v_{rs,k}=0$, the flow is complete and is given by the formula
\begin{equation}\label{eq:addflow}
 z\longmapsto z+\tau v_{rs,k}(q(z)),\qquad \tau\in\C.
\end{equation}
The identity $q(z+\tau v_{rs,k}(q(z)))^k=q(z)^k$ establishes both the flow law
and the inverse, the latter obtained by replacing $\tau$ with $-\tau$. A mode
with $k_r=k_s=0$ vanishes identically.

We introduce an accumulated displacement $h\in\C^n$. The corresponding augmented field on $\C^n_w\times\C^n_h$ is
\begin{equation}\label{eq:augfield}
 \widetilde V_{rs,k}(w,h)=(-i\diag(w)v_{rs,k}(w),\ v_{rs,k}(w)).
\end{equation}
Its complete entire flow is
\begin{equation}\label{eq:augflow}
 (w,h)\longmapsto
 \left((w_j e^{-i\tau v_{rs,k,j}(w)})_{j=1}^n,
          h+\tau v_{rs,k}(w)\right).
\end{equation}
Since $w^k$ is invariant along this flow, the velocity is constant. Because the
exponents are nonnegative, both formulas are entire on all of
$\C^n_w\times\C^n_h$, including the coordinate hyperplanes.

If a composition starts at $(q(z),0)$, induction over its factors gives the identity
\begin{equation}\label{eq:displacement}
 w_{\mathrm{final}}=q(A(z)),\qquad h_{\mathrm{final}}=A(z)-z,
\end{equation}
where $A$ is the corresponding composition of additive automorphisms \eqref{eq:addflow}. If $c$ depends holomorphically on further parameters, then the automorphisms and their inverses form jointly holomorphic families in those parameters.

The completeness of the frozen fields is decisive in the following
approximation: every approximating factor remains a global automorphism, rather
than merely a local flow map.

\begin{lemma}\label{lem:splitting}
Suppose finitely many entire vector fields $W_\nu(t,\zeta,y)$ depend holomorphically on $\zeta$, are continuously differentiable in real $t\in[0,1]$, and are complete for frozen $(t,\zeta)$. Assume that the trajectories of $\dot y=\sum_\nu W_\nu(t,\zeta,y)$ starting from a compact family of initial data, with parameters in a compact set, exist through time one and stay in a compact set. Then the time-one map of this equation is uniformly approximable by finite compositions of flows of the frozen fields, and these compositions and their inverses are holomorphic in $(\zeta,y)$.
\end{lemma}
\begin{proof}
Choose a compact neighborhood of all the exact trajectories, together with a
positive buffer around it. The fields, their first spatial derivatives, and
their time derivatives are uniformly bounded on this neighborhood. Writing
$a=1/N$ and using the successive frozen flows at $t_j=j/N$, we obtain
\[
 y\longmapsto y+a\sum_\nu W_\nu(t_j,\zeta,y)+O(a^2).
\]
The exact nonautonomous step admits the same expansion, uniformly on the
buffer. The accumulated error therefore satisfies
$e_{j+1}\le(1+Ca)e_j+Ca^2$ with $e_0=0$, whence $e_j\le C'a$ for $j\le N$.
Every factor is defined globally, since the frozen fields are complete, and its
dependence on $(\zeta,y)$, like that of its inverse, is holomorphic. For large
$N$ the accumulated error keeps all intermediate steps inside the buffer, and
the approximation follows.
\end{proof}

The next lemma is the conceptual core of the construction. Its two estimates
encode competing requirements: move the prescribed section along its path, but
move every point of $E_K$ by arbitrarily little. Multiplicative coordinates and
the accumulated displacement turn the latter, noncompact condition into control
near the compact set $K$.

\begin{lemma}\label{lem:transport}
Suppose $n\ge2$, $K\subset\C^n$ is compact and polynomially convex,
$U\subset\C^m$ is convex and open, and $L\Subset U$ is compact and convex. Assume that
$p:\C\times U\to\C^n$ is holomorphic and
\[
 p(t,\zeta)\notin E_K\qquad (0\le t\le1,\ \zeta\in L).
\]
For every $\epsilon>0$ there is a holomorphic family of automorphisms $A_\zeta$ of $\C^n$, $\zeta\in U$, with jointly holomorphic inverse and
\begin{align}
 \sup_{\zeta\in L,\ z\in E_K}|A_\zeta(z)-z|&<\epsilon,\label{eq:obstacle}\\
 \sup_{\zeta\in L}|A_\zeta(p(0,\zeta))-p(1,\zeta)|&<\epsilon.\label{eq:section}
\end{align}
\end{lemma}
\begin{proof}
If $L$ is empty, the identity family suffices; if $K$ is empty, so does the
translation by $p(1,\zeta)-p(0,\zeta)$. In the remaining case, set $B=\C\times
U$ and $J=[0,1]\times L$. The proof has four stages. We first prescribe local
Hamiltonian potentials near the obstacle and the section. We then approximate
them by polynomials with derivative control. The augmented flow converts compact
control near $K$ into displacement control on all of $E_K$, and a final
discretization by frozen complete flows produces the required automorphisms.

\smallskip\noindent\emph{Step 1: local Hamiltonian data.}
In $Z=B\times\C^n_w$, consider
\[
 P_0=J\times K,\quad
 G=\{(t,\zeta,w):w=q(p(t,\zeta))\},\quad
 Q_0=G|_J.
\]
The product $P_0$ is $\OO(Z)$-convex, because $J$ is convex and $K$ is
polynomially convex. The closed graph $G$ is biholomorphic to $B$, and $Q_0$ is
$\OO(G)$-convex. Avoidance gives $P_0\cap G=\varnothing$, so
Lemma~\ref{lem:union} makes $P_0\cup Q_0$ $\OO(Z)$-convex.

Set $v=\partial_t p$ and $a=q(p)$. The disjoint compact sets $P_0$ and $Q_0$
have positive distance, so we may choose disjoint open neighborhoods $N_P$ of
$P_0$ and $N_Q$ of $Q_0$. Each $a_j$ is nonzero, and the open set
\[
 \left\{(t,\zeta,w):\max_j\left|\frac{w_j}{a_j(t,\zeta)}-1\right|<\frac12\right\}
\]
contains $Q_0$. Intersecting $N_Q$ with this set leaves it disjoint from $N_P$
and still containing $Q_0$, while every ratio now lies in the fixed disc
$|\xi-1|<1/2$. On that disc we use the holomorphic logarithm normalized by
$\log1=0$. The $n-1$ potentials vanish on $N_P$ and have the following formulas
on $N_Q$:
\begin{align}
 H_{12}(t,\zeta,w)
   &=i\left(v_1\log\frac{w_2}{a_2}-v_2\log\frac{w_1}{a_1}\right),\label{eq:local12}\\
 H_{1j}(t,\zeta,w)&=-iv_j\log\frac{w_1}{a_1},\qquad 3\le j\le n.\label{eq:local1j}
\end{align}
These prescriptions give holomorphic potentials on the disjoint union
$N_P\cup N_Q$. On a
sufficiently small additive tube about $p(t,\zeta)$, their pullbacks by $q$ are
\[
 v_1(z_2-p_2)-v_2(z_1-p_1),\qquad -v_j(z_1-p_1),
\]
and their pair Hamiltonian fields sum to the constant translation $v$ in all
$n$ coordinates.

\smallskip\noindent\emph{Step 2: polynomial approximation with derivative
control.}

Fix compact $\OO(Z)$-convex neighborhoods $S_-,S_+$ of $P_0\cup Q_0$ with
$S_-\subset\operatorname{int}S_+\subset S_+\subset N_P\cup N_Q$. Each $S_\pm$
is the union of the disjoint compact parts $S_\pm\cap N_P$ and $S_\pm\cap N_Q$.
Fix a bounded neighborhood $N$ of $K$ with $J\times\overline N\subset
\operatorname{int}S_-\cap N_P$, and set
\[
 d=\operatorname{dist}(K,\C^n\setminus N)>0,\qquad
 C_N=\max_{w\in\overline N}\max_j|w_j|.
\]
Choose $\eta>0$ with $C_N\eta<d$ and $\eta<\epsilon/2$.

Each potential is holomorphic on an open neighborhood of the whole $S_+$, so
Lemma~\ref{lem:polyapprox} applies on $S_+$ and approximates all the potentials
simultaneously. For each potential $H$ and its polynomial approximant $\widehat
H$, Cauchy's estimates give
\[
 \sup_{S_-}|\partial_w^\alpha(\widehat H-H)|
 \le C_\alpha\sup_{S_+}|\widehat H-H|
\]
for the required fiber derivatives, with constants determined by the fixed
neighborhoods and their positive buffer. The sum of the approximating pair
Hamiltonian fields has velocity $\mathcal V(t,\zeta,w)$. The local potentials
vanish identically on $J\times\overline N$, and their fiber derivatives vanish
there as well; hence the Cauchy bound together with $|w_j|\le C_N$ in the Euler
operators $-iw_j\partial_{w_j}$ gives $|\mathcal V|<\eta$ on $J\times\overline
N$, including where some coordinates of $w$ vanish, once the errors
$\sup_{S_+}|\widehat H-H|$ are sufficiently small.

The approximating potentials have no negative powers, so their pullbacks
and vector fields are entire and single-valued in $z$.
On the tube about the section, the local logarithms have
the affine pullbacks displayed above. Fix $\delta>0$, before choosing the
approximation error, so that
\[
 (t,\zeta,q(p(t,\zeta)+u))\in\operatorname{int}S_-\cap N_Q
 \qquad ((t,\zeta)\in J,\ |u|\le\delta).
\]
Choose $0<\theta<\min\{\delta/2,\epsilon/2\}$ and make the derivative
approximation close enough that
\[
 |\mathcal V(t,\zeta,q(p(t,\zeta)+u))-v(t,\zeta)|<\theta
\]
throughout this fixed tube.

\smallskip\noindent\emph{Step 3: compact control in the augmented system.}

Consider the nonautonomous augmented system
\begin{equation}\label{eq:augsystem}
 \dot w=-i\diag(w)\mathcal V(t,\zeta,w),\qquad
 \dot h=\mathcal V(t,\zeta,w).
\end{equation}
For initial data $(w(0),h(0))=(w_0,0)$ with $w_0\in K$, the system gives,
as long as $w(t)$ has not left $N$,
\[
 |h(t)|\le\eta t,\qquad |w(t)-w_0|\le C_N\eta t.
\]
The second inequality prevents a first exit before time one. Since
$C_N\eta<d$, the trajectories with initial data in $K$ therefore remain in the
compact set
\[
 \{w:\dist(w,K)\le C_N\eta\}\times\{h:|h|\le\eta\}\subset N\times\C^n,
\]
and $|h(1)|<\epsilon/2$, uniformly in $\zeta\in L$, including for initial points
of $K$ on the coordinate hyperplanes.

For the section initial data $(q(p(0,\zeta)),0)$, the component equations
in \eqref{eq:augsystem} give $\frac{d}{dt}(w_j(t)e^{ih_j(t)})=0$. Since $h(0)=0$,
this yields
\[
 w(t)=q(p(0,\zeta)+h(t)).
\]
Set $e(t)=p(0,\zeta)+h(t)-p(t,\zeta)$, so that $e(0)=0$. On the maximal forward
interval of existence, while $t\le1$ and $|e(t)|<\delta$, the tube estimate
gives $|\dot e(t)|<\theta$, hence
\[
 |e(t)|\le\theta t<\delta/2.
\]
This rules out a first exit from the additive tube. For section trajectories,
\[
 w=q(p(t,\zeta)+e(t)),\qquad
 h=p(t,\zeta)-p(0,\zeta)+e(t),\qquad |e(t)|\le\theta.
\]
Since $J$ is compact, these formulas confine that family to a compact set as
well. All solutions therefore extend through time one, and the
section estimate gives
\[
 |h(1)-(p(1,\zeta)-p(0,\zeta))|<\epsilon/2.
\]
Lemma~\ref{lem:splitting} thus applies to both compact initial families.

\smallskip\noindent\emph{Step 4: frozen complete flows and the two final
estimates.}
Expand all polynomial potentials into monomials and apply
Lemma~\ref{lem:splitting} to their augmented fields \eqref{eq:augfield}, with
error less than $\epsilon/2$ in the displacement component. Since $q(z)\in K$
for every $z\in E_K$, all the corresponding augmented initial data lie in the
compact set $K\times\{0\}$. Identity~\eqref{eq:displacement} identifies the
final displacement with $A_\zeta(z)-z$, so the uniform estimate in the
augmented coordinates gives \eqref{eq:obstacle} on the entire, possibly
unbounded set $E_K$. For the section initial data, the same identity yields
\eqref{eq:section}.
\end{proof}

\section{Entire sprays}\label{sec:sprays}
This section closes the proof of Theorem~\ref{thm:general}. The first lemma
upgrades a continuous real-time deformation to a path that is entire in the
complex time variable, as required by Lemma~\ref{lem:transport}. A fixed
Fatou--Bieberbach model then turns the transported section into a normalized
entire spray.

\begin{lemma}\label{lem:path}
Suppose $Y\subset\C^n$ is open and path connected, $U\subset\C^m$ is convex,
$L\Subset U$ is compact and convex, and $f:U\to Y$ is holomorphic. For every
$p_*\in Y$ there exists a holomorphic map $p:\C\times U\to\C^n$ such that
\[
 p(0,\zeta)=f(\zeta),\quad p(1,\zeta)=p_*,\quad
 p(t,\zeta)\in Y\quad (t\in[0,1],\ \zeta\in L).
\]
\end{lemma}
\begin{proof}
If $L$ is empty, the affine path $p(t,\zeta)=(1-t)f(\zeta)+tp_*$ will do, so
assume that $L$ is nonempty and choose $\zeta_0\in L$. Contracting the source
from $\zeta$ to $\zeta_0$ inside $U$ and then following a fixed path in $Y$ from
$f(\zeta_0)$ to $p_*$ yields a continuous map $Q:[0,1]\times U\to\C^n$,
holomorphic in $\zeta$ at each fixed time, with $Q(0,\zeta)=f(\zeta)$ and
$Q(1,\zeta)=p_*$. Its image on the compact set $[0,1]\times L$ is compact in
$Y$, so we may choose $d_Y>0$ such that the $d_Y$-neighborhood of this image is
contained in $Y$. The Bernstein polynomials
\[
 p_N(s,\zeta)=\sum_{j=0}^N Q(j/N,\zeta)\binom Nj s^j(1-s)^{N-j}
\]
are holomorphic on $\C\times U$ and preserve the endpoints exactly. To estimate
them on $[0,1]\times L$, write $M_Q=\sup_{[0,1]\times L}|Q|$ and choose a
modulus of continuity $\omega$ for $Q$ in the time variable, uniform in
$\zeta\in L$;
such a modulus exists because $Q$ is continuous on the compact product
$[0,1]\times L$. For every $a>0$,
\[
 \sup_{[0,1]\times L}|p_N-Q|\le\omega(a)+\frac{M_Q}{2Na^2},
\]
since the mass of the Bernstein kernel over the indices satisfying $|s-j/N|>a$ is at most
$1/(4Na^2)$, its second moment being $s(1-s)/N\le1/(4N)$. Choosing $a$ with
$\omega(a)<d_Y/2$ and then $N$ so large that $M_Q/(2Na^2)<d_Y/2$ gives
$p_N([0,1]\times L)\subset Y$.
\end{proof}

For $M\in\mathbb R$, write
\[
 E_M^{(n)}=\{z\in\C^n:\im z_j\le M\text{ for every }j\}.
\]
Appendix~\ref{app:model} constructs a biholomorphism
\[
 \Phi_2:\C^2\longrightarrow B_2\subset\C^2\setminus E_0^{(2)},\qquad
 \Phi_2(0)=(i,i),\quad D\Phi_2(0)=I_2.
\]
The product map
\begin{equation}\label{eq:productmodel}
 \Phi(\xi)=\left(\Phi_2(\xi_1,\xi_2)+i(M+1)(1,1),
                  \xi_3,\ldots,\xi_n\right)
\end{equation}
is biholomorphic onto its image, satisfies $D\Phi(0)=I_n$, and is centered at
\[
 p_*=(i(M+2),i(M+2),0,\ldots,0).
\]
Its image avoids the intersection of the first two lower halfspaces
$\{\im z_1\le M+1,\ \im z_2\le M+1\}$ and therefore avoids $E_{M+1}^{(n)}$.

\begin{proof}[Proof of Theorem~\ref{thm:general}]
Nonemptiness and path connectedness follow from Lemma~\ref{lem:connected}. If
the source is empty, there is nothing to prove; if $E_K$ is empty, then
$Y_K=\C^n$, and affine sprays give the assertion. Suppose therefore that $E_K$
is nonempty and set $M=\log\max\{1,\max_{w\in K}\max_j|w_j|\}$. Since
$|e^{-iz_j}|=e^{\im z_j}$,
\[
 E_K\subset E_M^{(n)}.
\]
We use the fixed map \eqref{eq:productmodel}, whose center $p_*$ lies in $Y_K$.

Shrinking the original source if necessary, we may assume that it is a bounded
convex open set $U$, and we choose a compact convex $L$ with
\[
 L_0\subset\operatorname{int}L\subset L\Subset U\Subset U_0.
\]
By Lemma~\ref{lem:path} there is a holomorphic time path from $f$ to $p_*$ that
avoids $E_K$ for $t\in[0,1]$ and $\zeta\in L$. Applying
Lemma~\ref{lem:transport} with $\epsilon<1/4$, followed by the correction of
the center by the holomorphic translation
\[
 \widetilde A_\zeta(z)=A_\zeta(z)+p_*-A_\zeta(f(\zeta)),
\]
we obtain $\widetilde A_\zeta(f(\zeta))=p_*$ together with
\[
 \sup_{z\in E_K}|\widetilde A_\zeta(z)-z|<2\epsilon<1/2
 \quad (\zeta\in L).
\]
For $z\in E_K$ and every $j$,
\[
 \im\widetilde A_{\zeta,j}(z)
 \le\im z_j+|\widetilde A_{\zeta,j}(z)-z_j|
 <M+2\epsilon<M+1.
\]
Hence $\widetilde A_\zeta(E_K)\subset E_{M+1}^{(n)}$. On
$V=\operatorname{int}L$, consider
\[
 C_\zeta=D\widetilde A_\zeta(f(\zeta)),\qquad
 F(\zeta,\xi)=\widetilde A_\zeta^{-1}\bigl(\Phi(C_\zeta\xi)\bigr).
\]
The matrices $C_\zeta$ are holomorphic and invertible, and the inverse
automorphism family is jointly holomorphic, so $F$ is holomorphic on
$V\times\C^n$. Its image avoids $E_K$: indeed, $\Phi$ avoids $E_{M+1}^{(n)}$,
while $\widetilde A_\zeta(E_K)$ lies inside that set. Moreover
\[
 F(\zeta,0)=f(\zeta),\qquad
 D_\xi F(\zeta,0)=C_\zeta^{-1} I_n C_\zeta=I_n.
\]
Every fiber map is a composition of biholomorphisms onto their images, so the
normalization $D_\xi F(\zeta,0)=I_n$ makes $F$ a dominating spray, and
Kusakabe's convex spray criterion~\cite[Theorem 2.2]{Kus} gives the Oka property
of $Y_K$.

Since $q:Y_K\to X_K$ is an unramified holomorphic covering, the covering
invariance of the Oka property~\cite[Proposition 2.8]{F13} implies that $X_K$
is Oka as well.
\end{proof}

\section{Inversion charts and slab crosses}\label{sec:slabs}

The main device in this section is one-dimensional: translated inversion turns
the complement of an arbitrary connected planar domain into a compact
polynomially convex set. Applied coordinatewise, this places each punctured slab
chart in the setting of Theorem~\ref{thm:general}; Kusakabe's localization
theorem then assembles finitely many such charts.

\begin{lemma}\label{lem:inversion}
For a nonempty connected planar domain $U$ and $a\in U$, denote by $J_a$ the M\"obius transformation $z\mapsto1/(z-a)$ of $\PP^1$. Then
\[
 K_a=J_a(\PP^1\setminus U)
     =\{0\}\cup\{1/(x-a):x\in\C\setminus U\}
\]
is a compact polynomially convex subset of $\C$ containing $0$, and
\begin{equation}\label{eq:inversion}
 \C\setminus K_a=J_a(U\setminus\{a\}),\qquad
 w\in K_a\ \Longleftrightarrow\ a+1/w\notin U\quad(w\ne0).
\end{equation}
\end{lemma}
\begin{proof}
The compact set $\PP^1\setminus U$ does not contain the pole $a$, so its image
is compact and avoids infinity; since $\infty\notin U$, that image contains
$J_a(\infty)=0$. Bijectivity gives the two identities in \eqref{eq:inversion}.
The punctured planar domain $U\setminus\{a\}$ is path connected: a polygonal path
that passes through $a$ can be rerouted around $a$ by a small circle inside $U$.
Hence $\C\setminus K_a$ is connected. Since the planar polynomial hull consists
exactly of $K_a$ together with the bounded components of its complement, $K_a$
is polynomially convex. For $U=\C$ one has $K_a=\{0\}$.
\end{proof}

Thus connectedness of $U$ is used at one precise point: it makes the complement
of the inverted compact set connected, which in one variable is exactly the
condition needed for polynomial convexity.

\begin{proof}[Proof of Theorem~\ref{thm:slabs}]
Write $\mathcal S=\mathcal S(U_1,\ldots,U_n)$ and fix $a_j\in U_j$. Consider
\[
 \mathcal S_a=\mathcal S\setminus\bigcup_{j=1}^n\{z_j=a_j\},\qquad
 K=\prod_{j=1}^n J_{a_j}(\PP^1\setminus U_j).
\]
By Lemma~\ref{lem:inversion} each compact factor is polynomially convex, and
hence so is their product: a point outside the product has a coordinate outside
the corresponding factor, and a one-variable polynomial in that coordinate
separates it from the product. Coordinatewise inversion is a biholomorphism of
the ambient chart with the coordinate hyperplanes $z_j=a_j$ removed onto
$(\C^*)^n$, and \eqref{eq:inversion} gives the identity
\[
 (J_{a_1},\ldots,J_{a_n})(\mathcal S_a)
   =(\C^*)^n\setminus\bigl(K\cap(\C^*)^n\bigr)=X_K.
\]
Theorem~\ref{thm:general} therefore makes $\mathcal S_a$ Oka.

We now choose two distinct points $a_{j,0},a_{j,1}\in U_j$ for each $j$. The
resulting $2^n$ charts cover $\mathcal S$, since for every $z$ one may choose,
independently in each coordinate, an $a_{j,\varepsilon_j}$ different from $z_j$.
Each chart is Zariski open in $\mathcal S$: its complement in $\mathcal S$ is the
zero set of $\prod_j(z_j-a_{j,\varepsilon_j})$, hence a closed complex
subvariety of $\mathcal S$.
The manifold $\mathcal S$ is connected, its $n$ connected open slabs having the
nonempty common intersection $\prod_jU_j$, so $\mathcal S$ is Oka by Kusakabe's
localization theorem~\cite[Theorem 1.4]{Kus}.
\end{proof}

When $n=1$, the corresponding domain $\Omega_1=\D$ is not Oka.

\subsection{Removing finitely many points from a slab cross}

Finite point removal illustrates a useful stability of the method. In an
inversion chart, deleting finitely many points amounts to adjoining finitely
many points to the compact polynomially convex obstacle. The elementary lemma
below preserves polynomial convexity under that operation.

\begin{corollary}\label{cor:finite-punctures}
If $n\ge2$, $U_1,\ldots,U_n\subset\C$ are nonempty connected open sets, and $P\subset\mathcal S(U_1,\ldots,U_n)$ is finite, then $\mathcal S(U_1,\ldots,U_n)\setminus P$ is Oka.
\end{corollary}

\begin{lemma}\label{lem:finite-adjunction}
If $K\subset\mathbb C^n$ is compact and polynomially convex and $Q\subset\mathbb C^n$ is finite, then $K\cup Q$ is polynomially convex.
\end{lemma}
\begin{proof}
Fix $x\notin K\cup Q$. Choose affine linear polynomials $\ell_q$ with
$\ell_q(q)=0$ and $\ell_q(x)\ne0$, and set $h=\prod_{q\in Q}\ell_q$, with
$h=1$ if $Q=\varnothing$. If $K=\varnothing$, $h$ separates $x$ from $Q$.
Otherwise choose a polynomial $p$ with $|p(x)|>\sup_K|p|$. Since
\[
 \frac{\sup_K|hp^N|}{|h(x)p(x)^N|}
 \le\frac{\sup_K|h|}{|h(x)|}
       \left(\frac{\sup_K|p|}{|p(x)|}\right)^N\longrightarrow0,
\]
the polynomial $hp^N$, which vanishes on $Q$, separates $x$ from $K\cup Q$
for large $N$.
\end{proof}

\begin{proof}[Proof of Corollary~\ref{cor:finite-punctures}]
Write $\mathcal S=\mathcal S(U_1,\ldots,U_n)$. Use the $2^n$ inversion charts
$\mathcal S_a$ and the compact polynomially
convex sets $K$ from the proof of Theorem~\ref{thm:slabs}. In each chart let
$I_a=(J_{a_1},\ldots,J_{a_n})$ and $Q_a=I_a(P\cap\mathcal S_a)$. Then $Q_a$ is
finite, is disjoint from $K$, and
\[
 I_a(\mathcal S_a\setminus P)
   =(\C^*)^n\setminus(K\cup Q_a)=X_{K\cup Q_a}.
\]
Lemma~\ref{lem:finite-adjunction} and Theorem~\ref{thm:general} show that every
$\mathcal S_a\setminus P$ is Oka. These sets form an analytic-Zariski open
cover of $\mathcal S\setminus P$. Since deleting finitely many points does not
disconnect the connected open manifold $\mathcal S$, Kusakabe's localization
theorem~\cite[Theorem~1.4]{Kus} applies again and makes $\mathcal S\setminus P$
Oka.
\end{proof}

Alarc\'on and Forstneri\v c~\cite[Corollary~2.9]{AF} also established a
finite-removal result in the Oka-1 category. For punctured slab crosses,
Corollary~\ref{cor:finite-punctures} gives the full Oka property.

\section{Periodic Euclidean balls}\label{sec:periodic}

Coordinatewise inversion turns the periodic ball obstacle into a relatively
compact set whose closure meets the coordinate hyperplanes;
Theorem~\ref{thm:general} applies once that closure is known to be polynomially
convex.

\subsection{Polynomially convex compactification}

Throughout this section the lattices and $r$ satisfy the hypotheses of
Theorem~\ref{thm:balls}. For each $j$, consider the planar obstacle and its
complement
\[
 A_j=\bigcup_{\lambda\in\Lambda_j}\Dclosed(\lambda,r),
 \qquad U_j=\C\setminus A_j.
\]
Here $\Dclosed(\lambda,r)$ denotes a closed planar disc. The discs are locally
finite and have mutual distance at least $\ell_j-2r>0$, so $A_j$ is closed. The
annulus $r<|z|<\ell_j-r$ lies in $U_j$, and in particular $U_j$ is nonempty.

To see that $U_j$ is path connected, take two points of $U_j$. Since their
distances from the closed set $A_j$ are positive, a sufficiently small increase
of the radius from $r$ to some $s>r$ with $2s<\ell_j$ leaves both points outside
every enlarged disc. A line segment joining the two points meets only finitely
many of these discs; travelling around the boundary circle of each enlarged disc
in place of the portion of the segment inside it gives a path in $U_j$: the
enlarged discs are pairwise
disjoint, and their boundary circles avoid all the original discs of radius $r$.

For $a=(a_1,\ldots,a_n)\in\prod_j U_j$, the map
\[
 I_a:Z_a=\prod_{j=1}^n(\C\setminus\{a_j\})\longrightarrow(\C^*)^n,
 \qquad I_a(z)=\left(\frac1{z_1-a_1},\ldots,\frac1{z_n-a_n}\right)
\]
is a biholomorphism. The closed, locally finite union
\[
 E=\bigcup_{\lambda\in\Lambda}\B(\lambda,r)
\]
is contained in $Z_a$, because its $j$th coordinate projection is $A_j$.

\begin{lemma}\label{lem:compactification}
For every $a\in\prod_j U_j$, the closure $K=\overline{I_a(E)}$ in $\C^n$ is compact and polynomially convex, and
\begin{equation}\label{eq:torus-intersection}
 K\cap(\C^*)^n=I_a(E).
\end{equation}
\end{lemma}

\begin{proof}
The argument has four stages. First, the inverted obstacle is placed in a
compact polynomially convex product envelope. Second, all strata on the
coordinate hyperplanes are identified, together with the projections needed
later. Third, each stratum with prescribed nonzero coordinates is shown to be
polynomially convex. Finally, approximate component indicators localize a point
of the polynomial hull to one stratum, where the preceding identities force it
back into $K$.

\smallskip\noindent\emph{Step 1: a compact product envelope.}
For $\lambda\in\Lambda_j$, write
\[
 D_{j,\lambda}=\left\{\frac1{z-a_j}:|z-\lambda|\le r\right\},
 \qquad
 P_j=\{0\}\cup\bigcup_{\lambda\in\Lambda_j}D_{j,\lambda},
 \qquad P=\prod_{j=1}^nP_j.
\]
Since $a_j\notin\Dclosed(\lambda,r)$, its image $D_{j,\lambda}$ under the
M\"obius map $z\mapsto1/(z-a_j)$ is a closed Euclidean disc not containing
zero, and these discs are pairwise disjoint. Moreover,
\[
 \sup_{w\in D_{j,\lambda}}|w|
 \le\frac1{|\lambda-a_j|-r}\longrightarrow0
 \qquad (|\lambda|\longrightarrow\infty).
\]
Since $P_j=J_{a_j}(\PP^1\setminus U_j)$, Lemma~\ref{lem:inversion} shows
that $P_j$ is compact and polynomially convex.

Each disc $D_{j,\lambda}$ is separated from $P_j\setminus D_{j,\lambda}$ by a
positive distance. Indeed, its distance from zero is positive, so all but
finitely many of the other discs lie in a sufficiently small neighborhood of
zero, while the remaining finitely many compact discs are disjoint from
$D_{j,\lambda}$ and hence at positive distance from it. In particular,
$D_{j,\lambda}$ is an isolated component of $P_j$.

The product $P$ is polynomially convex by the coordinate separation argument
used in the proof of Theorem~\ref{thm:slabs}.

Since $I_a(E)\subset P$ with $P$ compact and polynomially convex, $K$ is compact
and $\widehat K\subset\widehat P=P$, where $\widehat K$ denotes the polynomial
hull. If $I_a(z^\nu)\to w\in(\C^*)^n$ with $z^\nu\in E$, then
\[
 z^\nu\longrightarrow a+1/w
   =\left(a_j+\frac1{w_j}\right)_{j=1}^n.
\]
Closedness of $E$ proves \eqref{eq:torus-intersection}.

\smallskip\noindent\emph{Step 2: boundary strata and their projections.}
For a nonempty set $I\subset\{1,\ldots,n\}$ and a tuple
$\lambda_I\in\prod_{i\in I}\Lambda_i$, write
\[
 B_{I,\lambda_I}
 =\left\{z_I\in\C^I:\sum_{i\in I}|z_i-\lambda_i|^2\le r^2\right\},
 \qquad
 Q_{I,\lambda_I}=I_{a,I}(B_{I,\lambda_I}),
\]
where $I_{a,I}$ denotes the coordinatewise inversion restricted to the
coordinates in $I$. Every coordinate of every point of $Q_{I,\lambda_I}$ is
nonzero.

For $w\in P\setminus\{0\}$, take $I=\{i:w_i\ne0\}$ and denote by $\lambda_i$
the unique label with $w_i\in D_{i,\lambda_i}$. Then
\begin{equation}\label{eq:stratum}
 w\in K\quad\Longleftrightarrow\quad w_I\in Q_{I,\lambda_I}.
\end{equation}
For the forward implication, choose $z^\nu\in\B(\lambda^\nu,r)$ with
$I_a(z^\nu)\to w$. Isolation of $D_{i,\lambda_i}$ forces
$\lambda_i^\nu=\lambda_i$ for all sufficiently large $\nu$ and for each $i\in I$,
and since $I$ is finite these identities hold simultaneously for large $\nu$.
Passing to the limit in
\[
 \sum_{i\in I}|z_i^\nu-\lambda_i|^2\le r^2
\]
gives $a_I+1/w_I\in B_{I,\lambda_I}$.

Conversely, suppose that $w_I\in Q_{I,\lambda_I}$ and keep $z_I=a_I+1/w_I$
fixed. For each $j\notin I$, choose $\lambda_j^\nu\in\Lambda_j$ with
$|\lambda_j^\nu|\to\infty$ and set $z_j^\nu=\lambda_j^\nu$. The center
$(\lambda_I,\lambda_{I^c}^\nu)$ belongs to the full product lattice, and
\[
 |(z_I,z_{I^c}^\nu)-(\lambda_I,\lambda_{I^c}^\nu)|^2
   =\sum_{i\in I}|z_i-\lambda_i|^2\le r^2.
\]
Thus $(z_I,z_{I^c}^\nu)\in E$, and its inversion tends to $w$. When $I$ is the
full index set the constant sequence suffices, and choosing all coordinates to
be lattice centers tending to infinity likewise shows that $0\in K$.

Besides \eqref{eq:stratum} we shall need a projection identity. For every
nonempty $I$ and every tuple $\lambda_I$, the compact subset
\[
 K_{I,\lambda_I}
   =\{w\in K:w_i\in D_{i,\lambda_i}\text{ for every }i\in I\}
\]
satisfies
\begin{equation}\label{eq:projection}
 \pr_I(K_{I,\lambda_I})=Q_{I,\lambda_I}.
\end{equation}
Indeed, for any $w\in K_{I,\lambda_I}$ the stabilization and limiting argument
above, applied in the coordinates of $I$, gives $w_I\in Q_{I,\lambda_I}$
whatever the remaining coordinates are. Conversely, if $q\in Q_{I,\lambda_I}$,
then by \eqref{eq:stratum} the point with coordinates $q$ on $I$ and zero on
$I^c$ belongs to $K$, and since $q_i\in D_{i,\lambda_i}$ for every $i\in I$,
this point lies in $K_{I,\lambda_I}$. Both inclusions in \eqref{eq:projection}
follow.

\smallskip\noindent\emph{Step 3: polynomial convexity of each stratum.}
The set $Q_{I,\lambda_I}$ is contained in the polynomially convex product of discs
\[
 D_I=\prod_{i\in I}D_{i,\lambda_i}\subset(\C^*)^I,
\]
so a point outside $D_I$ is already separated from $Q_{I,\lambda_I}$ by a
polynomial. Suppose then that $p\in D_I\setminus Q_{I,\lambda_I}$. Then
$a_I+1/p$ lies outside the compact convex ball $B_{I,\lambda_I}$, and since a
compact convex set is polynomially convex, some polynomial $q$ satisfies
\[
 |q(a_I+1/p)|>\sup_{z_I\in B_{I,\lambda_I}}|q(z_I)|.
\]
The function $h(w_I)=q(a_I+1/w_I)$ is holomorphic on a neighborhood of $D_I$,
because $D_I$ avoids every coordinate hyperplane. Oka--Weil approximation on
$D_I$ therefore produces polynomial approximants to $h$, uniformly on $D_I$,
and an error below a third of the gap in the inequality above preserves the
separation of $p$ from $Q_{I,\lambda_I}$. Hence
$Q_{I,\lambda_I}$ is polynomially convex.

\smallskip\noindent\emph{Step 4: localization of the polynomial hull.}
Fix $p\in\widehat K\subset P$. The case $p=0$ was settled above, so assume
$p\ne0$. Write $I=\{i:p_i\ne0\}$ for the set of its nonzero coordinates, and
denote by $\lambda_i$ the unique label with
$p_i\in D_{i,\lambda_i}$. For $i\in I$,
the function $h_i$ on $P_i$ that equals $1$ on $D_{i,\lambda_i}$ and $0$ on its
complement extends holomorphically to a neighborhood of $P_i$: it suffices to
take disjoint open neighborhoods of these two compact sets and to prescribe the
corresponding constant values on them. Oka--Weil approximation then supplies
one-variable polynomials $h_{i,\nu}$ converging uniformly to $h_i$ on $P_i$, so
the polynomials
\[
 H_\nu(w)=\prod_{i\in I}h_{i,\nu}(w_i)
\]
converge uniformly on $P$ to $h(w)=\prod_{i\in I}h_i(w_i)$. This limit equals
$1$ at $p$, equals $1$ on $K_{I,\lambda_I}$, and vanishes on
$K\setminus K_{I,\lambda_I}$.

For any polynomial $Q$ in the coordinates indexed by $I$, the hull inequality
gives
\[
 |H_\nu(p)Q(p_I)|\le\sup_{w\in K}|H_\nu(w)Q(w_I)|.
\]
Uniform convergence on the compact set $P$, followed by \eqref{eq:projection},
yields
\[
 |Q(p_I)|\le\sup_{w\in K_{I,\lambda_I}}|Q(w_I)|
             =\sup_{q\in Q_{I,\lambda_I}}|Q(q)|.
\]
Since $Q_{I,\lambda_I}$ is polynomially convex, the inequality above gives
$p_I\in Q_{I,\lambda_I}$, and \eqref{eq:stratum} then gives $p\in K$. Hence
$\widehat K=K$.
\end{proof}

The lemma is the only application-specific convexity ingredient in
Theorem~\ref{thm:balls}. The functions $H_\nu$ in the last step do not further
approximate the obstacle; they localize the hull inequality to the unique block
$K_{I,\lambda_I}$ selected by the nonzero coordinates of $p$.

\subsection{Oka localization}

\begin{proof}[Proof of Theorem~\ref{thm:balls}]
For every $a\in\prod_jU_j$, Lemma~\ref{lem:compactification} provides a compact
polynomially convex set $K_a=\overline{I_a(E)}$ together with a biholomorphism
\[
 I_a:\Omega\cap Z_a\longrightarrow(\C^*)^n\setminus K_a.
\]
Theorem~\ref{thm:general} therefore applies to each $\Omega\cap Z_a$. Choose two
distinct points $a_{j,0},a_{j,1}\in U_j$ for each $j$. As in the proof of
Theorem~\ref{thm:slabs}, the $2^n$ sets
\[
 W_\varepsilon=\Omega\setminus\bigcup_{j=1}^n\{z_j=a_{j,\varepsilon_j}\},
 \qquad\varepsilon\in\{0,1\}^n,
\]
form an analytic-Zariski open cover of $\Omega$.

The nonempty product $\prod_jU_j$ lies in $\Omega$. Given two points of
$\Omega$, their positive distances from the closed set $E$ allow us to enlarge
the radius to some $s$ with $r<s<\ell/2$ while keeping both points outside all
the enlarged balls. A segment joining the two points meets only finitely many
such balls, and replacing each portion inside one of them by a path on its
boundary sphere yields a path in $\Omega$, since these spheres are path
connected and avoid every original ball. The analytic-Zariski open cover above
is therefore a cover by Oka domains, and Kusakabe's localization
theorem~\cite[Theorem 1.4]{Kus} makes $\Omega$ Oka.
\end{proof}

\section{Products of elliptic curves}\label{sec:elliptic}

\subsection{Product holes}

This theorem is the quotient form of Theorem~\ref{thm:slabs}. Under the product
universal covering, the full inverse image of the product-hole domain is exactly
a slab cross; connectedness of the planar lifts is the precise topological
hypothesis that makes the reduction work.

\begin{theorem}
\label{thm:elliptic-product-holes}
Suppose $n\ge2$. For $1\le j\le n$, suppose
$E_j=\mathbb C/\Lambda_j$ is an elliptic curve with quotient map
$\pi_j:\mathbb C\to E_j$, and $D_j\subset E_j$ is a nonempty connected open
set whose full inverse image $U_j=\pi_j^{-1}(D_j)$ is connected. Then
\[
 X=\bigcup_{j=1}^n
    \bigl(E_1\times\cdots\times E_{j-1}\times D_j
                 \times E_{j+1}\times\cdots\times E_n\bigr)
\]
is Oka. Equivalently, writing $K_j=E_j\setminus D_j$, the domain
\[
 (E_1\times\cdots\times E_n)
       \setminus(K_1\times\cdots\times K_n)
\]
is Oka.
\end{theorem}
\begin{proof}
The product universal covering $\pi=\pi_1\times\cdots\times\pi_n$ has full
inverse image
\[
 \pi^{-1}(X)=\mathcal S(U_1,\ldots,U_n).
\]
Each $U_j$ is a nonempty connected planar domain, so
Theorem~\ref{thm:slabs} makes the inverse image Oka. Its image $X$ is connected,
and the restriction of $\pi$ is an unramified holomorphic covering onto $X$.
By the covering invariance of the Oka property~\cite[Proposition 2.8]{F13},
$X$ is Oka.
\end{proof}

The connected-lift hypothesis can be tested topologically: for a nonempty
connected open set $D\subset E$, the full lift to $\mathbb C$ is
connected if and only if
\[
 \pi_1(D)\longrightarrow\pi_1(E)=\Lambda
\]
is surjective. Indeed, the components of the full lift are indexed by the
cosets of the image subgroup in the deck group $\Lambda$, so the lift is
connected exactly when the image is all of $\Lambda$.

\begin{corollary}\label{cor:closed-disc-product-holes}
Suppose $n\ge2$ and $E_1,\ldots,E_n$ are elliptic curves. If each
$K_j\subset E_j$ is a closed embedded disc, then
\[
 (E_1\times\cdots\times E_n)\setminus(K_1\times\cdots\times K_n)
\]
is Oka. The same conclusion holds when each $K_j$ is a finite union of pairwise disjoint closed embedded discs.
\end{corollary}

Here a closed embedded disc means the image of the closed unit disc under a
topological embedding into $E_j$; no boundary regularity is assumed.
The point is that the connected-lift hypothesis in
Theorem~\ref{thm:elliptic-product-holes} becomes a concrete fundamental-group
test for disc complements.

\begin{proof}[Proof of Corollary~\ref{cor:closed-disc-product-holes}]
Fix two points of $D_j=E_j\setminus K_j$. By the Jordan--Schoenflies theorem
for surfaces, the boundary of each embedded disc $B$ has an exterior collar:
there is an embedding
\[
 c_B:\partial B\times[0,\varepsilon_B]\longrightarrow
 E_j\setminus\operatorname{int}B,\qquad c_B(x,0)=x,
\]
whose image meets $B$ precisely in $\partial B$. Choose pairwise disjoint
open neighborhoods of the finitely many discs that avoid the two fixed
points. By shortening the collars, their images lie in these neighborhoods.
Attaching each collar to its disc gives a closed embedded disc
\[
 B^+=B\cup c_B(\partial B\times[0,\varepsilon_B]),
 \qquad B\subset\operatorname{int}B^+.
\]
Thus the enlarged discs are pairwise disjoint and avoid both points.
Rerouting a path between the points along the enlarged boundary circles
shows that $D_j$ is connected.

Write $M_j\subset D_j$ for the connected compact
surface obtained by removing the interiors of the enlarged discs. Filling in its
boundary circles with discs recovers $E_j$, so van Kampen's theorem gives a
surjection $\pi_1(M_j)\to\pi_1(E_j)$. The inclusions $M_j\subset D_j\subset E_j$
factor this surjection through $\pi_1(D_j)$, whose image is therefore all of
$\pi_1(E_j)$. The covering-space criterion above now shows that each
$U_j=\pi_j^{-1}(D_j)$ is connected, and
Theorem~\ref{thm:elliptic-product-holes} applies.
\end{proof}

Theorem~\ref{thm:elliptic-product-holes} also applies when each compact $K_j$
lies in a coordinate disc and $E_j\setminus K_j$ is connected, since the two
generators of $\pi_1(E_j)$ can then be represented outside a slightly larger
disc. Containment in a disc without connectedness is insufficient for this
argument, since a simple closed curve separates off an interior component.

\subsection{A closed metric ball}

For the flat product torus $T$ and the lattices fixed in the Introduction, the
full inverse image of a closed ball whose radius is below the injectivity radius
is exactly the periodic obstacle of Theorem~\ref{thm:balls}.

\begin{proof}[Proof of Corollary~\ref{cor:torus}]
By translation we may take $c=0$, and denote the quotient map by
$\pi:\C^n\to T$. Two points of $\B(0,r)$ differ by a vector of length at most
$2r<\ell$, so
their difference cannot be a nonzero lattice vector; the same argument applies
to a slightly larger radius. Hence $\pi$ embeds a neighborhood of $\B(0,r)$,
and $\pi(\B(0,r))=\B_T(0,r)$. Moreover,
\[
 \pi^{-1}(\B_T(0,r))
 =\pi^{-1}(\pi(\B(0,r)))
 =\bigcup_{\lambda\in\Lambda}\B(\lambda,r)=E.
\]
Thus $\pi$ restricts to an unramified holomorphic covering from the Oka manifold
$\Omega$ onto $T\setminus\B_T(0,r)$. The covering invariance of the Oka
property~\cite[Proposition 2.8]{F13} therefore implies that
$T\setminus\B_T(0,r)$ is Oka as well.
\end{proof}

\appendix
\section{An explicit two-dimensional fixed model}\label{app:model}
This appendix supplies the single fixed Fatou--Bieberbach model used in
\eqref{eq:productmodel}. Composing two explicit shears with an affine
contraction gives an automorphism that maps the lower quadrant into its
interior and fixes an attracting point outside it; its attracting basin avoids
the obstacle and is biholomorphic to $\C^2$.

Write $E_0^{(2)}=\{\im z_1\le0,\im z_2\le0\}$ and $p=(i,i)$. The polynomials
\[
 Q_1(\xi)=3(\xi/e)^5(6-5\xi/e),\qquad
 Q_2(\xi)=3(\xi/e^4)(2-\xi/e^4)
\]
take the value $3$ at $e$ and $e^4$, respectively, and their derivatives vanish
at those points. Moreover
\begin{align*}
 \sup_{|\xi|\le1}|Q_1(\xi)|
 &\le3e^{-5}(6+5e^{-1})<\tfrac12,\\
 \sup_{|\xi|\le e^{1/2}}|Q_2(\xi)|
 &\le3e^{-7/2}(2+e^{-7/2})<\tfrac12.
\end{align*}
Consider the shears
\[
 A_1(z)=(z_1+iQ_1(e^{-iz_2}),z_2),\qquad
 A_2(z)=(z_1,z_2+iQ_2(e^{-iz_1})),
\]
\[
 \Psi(z)=\tfrac12(A_2\circ A_1)(z)-i(1,1).
\]
Each shear, and therefore $\Psi$, is an entire automorphism. The successive
images of $p$ under the two shears are $(4i,i)$ and $(4i,4i)$, so $\Psi(p)=p$
and $D\Psi(p)=\frac12I_2$. For $z\in E_0^{(2)}$, the first shear gives
$\im(A_1z)_1<1/2$ and leaves the second coordinate unchanged, while the second
polynomial is evaluated at modulus at most $e^{1/2}$, so its output has
$\im(A_2A_1z)_2<1/2$ without changing the first coordinate. Hence
\[
 \Psi(E_0^{(2)})\subset\{\im z_1<-3/4,\ \im z_2<-3/4\}
 \subset E_0^{(2)}.
\]
The attracting basin of $p$ avoids this closed forward-invariant set: the
iterates of a point of that set remain in it, whereas those of a basin point
converge to $p$.

The eigenvalues of $D\Psi(p)$ are both $1/2$, so
$|1/2|^2<|1/2|<1$. Rosay--Rudin's attracting-basin theorem
\cite[Theorem 9.1]{RR} therefore gives a biholomorphism from the basin $B_2$ of
$p$ onto $\C^2$. Write $\psi(u)=\Psi(p+u)-p$. In these translated coordinates,
the theorem's map is
$H=\lim_{N\to\infty}2^N\psi^N$, with $H(0)=0$ and $DH(0)=I_2$.
Thus its inverse, translated back by $p$, gives
\[
 \Phi_2:\C^2\longrightarrow B_2\subset\C^2\setminus E_0^{(2)},\qquad
 \Phi_2(0)=p,\quad D\Phi_2(0)=I_2,
\]
and this is the biholomorphism used in \eqref{eq:productmodel}.

\section*{Funding}
Partial support for S.-Y.~Xie came from the National Key R\&D Program of China through grants 2023YFA1010500 and 2021YFA1003100, from the National Natural Science Foundation of China through grants 12288201 and 12471081, and from the Xiaomi Young Talents Program.

\section*{AI use statement}
The authors developed the mathematical ideas and arguments. AI assistance was used for manuscript preparation. Responsibility for reviewing and approving all content rests with the authors.

\bigskip
\begingroup
\small
\noindent\textsc{(Yun-Heng Du) Academy of Mathematics and Systems Science, Chinese
Academy of Sciences, Beijing 100190, China.}\par
\noindent Email address: \href{mailto:duyunheng@amss.ac.cn}{\nolinkurl{duyunheng@amss.ac.cn}}\par
\medskip
\noindent\textsc{(Bin Guo) Academy of Mathematics and Systems Science, Chinese
Academy of Sciences, Beijing 100190, China.}\par
\noindent Email address: \href{mailto:guobin181@mails.ucas.ac.cn}{\nolinkurl{guobin181@mails.ucas.ac.cn}}\par
\medskip
\noindent\textsc{(Peng-Chao Wang) Academy of Mathematics and Systems Science, Chinese
Academy of Sciences, Beijing 100190, China.}\par
\noindent Email address: \href{mailto:wangpengchao26@mails.ucas.ac.cn}{\nolinkurl{wangpengchao26@mails.ucas.ac.cn}}\par
\medskip
\noindent\textsc{(Song-Yan Xie) State Key Laboratory of Mathematical Sciences,
Academy of Mathematics and Systems Science, Chinese Academy of Sciences,
Beijing 100190, China; School of Mathematical Sciences, University of Chinese
Academy of Sciences, Beijing 100049, China.}\par
\noindent Email address: \href{mailto:xiesongyan@amss.ac.cn}{\nolinkurl{xiesongyan@amss.ac.cn}}\par
\endgroup
\end{document}